\documentclass{arxiv}

\usepackage{graphicx, shuffle, amsaddr}

\newtheorem{notation}[theorem]{Notation}
\def\letter#1{\mathtt{#1}}

\DeclareMathOperator{\id}{id}
\DeclarePairedDelimiter{\dlrp}{(\mskip-4mu(}{)\mskip-4mu)}
\DeclareMathOperator{\Area}{Area}
\DeclareMathOperator{\area}{area}

\DeclareMathOperator{\darea}{\letter{area}}
\def\qsucc{\mathbin{\dot{\succ}}}
\def\qprec{\mathbin{\dot{\prec}}}
\def\qdot{\diamond}

\newcommand\DS{\operatorname{ISS}} 
\newcommand\proj{\operatorname{proj}}

\newcommand\field{\mathbb F}
\newcommand\NZero{{\N}}
\newcommand\NOne{{\N_+}}

\begin{document}

\title{Time-warping invariants of multidimensional time series}

\author[J. Diehl]{Joscha Diehl}
\address{Universit\"at Greifswald, Institut f\"ur Mathematik und Informatik, Walther-Rathenau-Str.~47, 17489 Greifswald, Germany.}
\email{joscha.diehl@uni-greifswald.de}
\author[K. Ebrahimi-Fard]{Kurusch Ebrahimi-Fard}
\address{Department of Mathematical Sciences, NTNU, 7491 Trondheim, Norway.}
\email{kurusch.ebrahimi-fard@ntnu.no}
\author[N. Tapia]{Nikolas Tapia}
\address{Weierstra{\ss}-Institut Berlin, Mohrenstr.~39, 10117 Berlin, Germany\\Technische Universität Berlin, Str. des 17.~Juni 136, 10623 Berlin, Germany.}
\email{tapia@wias-berlin.de}

\date{\today}
\maketitle

\begin{abstract}
In data science, one is often confronted with a time series representing measurements of some quantity of interest.  Usually, in a first step, features of the time series need to be extracted.  These are numerical quantities that aim to succinctly describe the data and to dampen the influence of noise.  

In some applications, these features are also required to satisfy some invariance properties.  In this paper, we concentrate on time-warping invariants. We show that these correspond to a certain family of iterated sums of the increments of the time series, known as quasisymmetric functions in the mathematics literature.  We present these invariant features in an algebraic framework, and we develop some of their basic properties.
\end{abstract}

\section{Motivation}
\label{sec:motivation}

Given a discrete time series
\begin{align*}
	x = (x_0, x_1, \ldots, x_N) \in (\R^d)^N,
\end{align*}
where $N \ge 1$ is some arbitrary time horizon, our foremost, and original, motivation stems from the desire to extract features from $x$ that are invariant to \emph{time warping}.

The precise definition of the latter will be given in Section \ref{sec:inv}, but Figure \ref{fig:timeWarping} illustrates what we mean by time warping: the time series is allowed to ``stand still'' or to ``stutter'' (this term is used in \cite{Yi1998Efficient}), which means that $x$ has repetitions of values at consecutive time steps (here at time $t=3$).

\begin{remark}
In this section we consider the notationally simpler case $d=1$, that is, when $x \in \R^N$.
\end{remark}

\begin{figure}[ht]
  \centering
  \begin{minipage}{0.3\textwidth}
      \centering
      \includegraphics[width=1.3\textwidth]{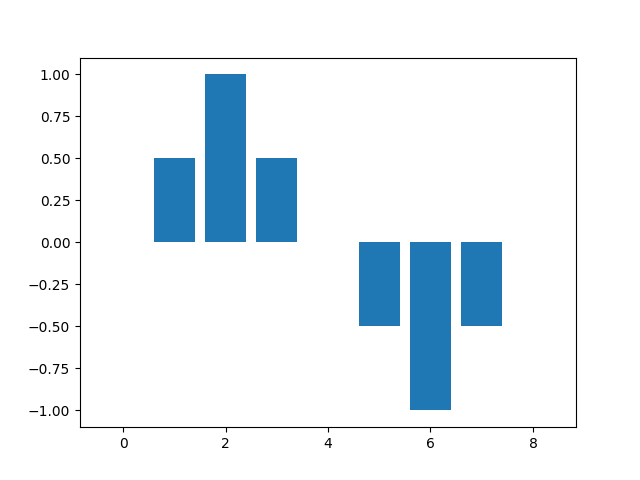}
  \end{minipage}
\qquad \qquad   $\mapsto$\
  \begin{minipage}{0.3\textwidth}
      \centering
      \includegraphics[width=1.3\textwidth]{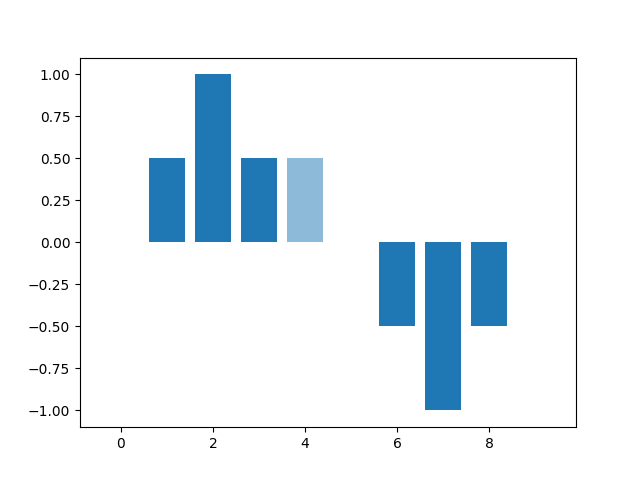}
  \end{minipage}
  \caption{Example of time warping in the case of a discrete time series in $d=1$ dimensions.}
  \label{fig:timeWarping}
\end{figure}

Our interest is prompted, on the one hand by the extensive literature on the \emph{dynamic time warping} (DTW) distance \cite{Berndt1994}, a distance on discrete time series that is invariant to time warping.
In \cite{Yi1998Efficient}
it is stated that
\begin{center}
  \emph{``the time warping distance \ldots does not lead to any natural features''}.
\end{center}
Our work aims to provide those missing ``natural'' features.

On the other hand the following example illustrates where such invariant features
will become useful.

\begin{example}
  \label{ex:inference}
  Assume that there is a deterministic time series $x \in \R^N$
  which models some ``prototype'' evolution of a quantity,
  say the prototype heartbeat in a patient's ECG.
  This prototype is unknown, but
  one records a lot of samples of it \emph{run at different speeds}
  and contaminated by noise (compare \cite{Bigot2013}).
  A model for these observations is then
  \begin{align*}
	y^{(\ell)}_n = x_{h^{(\ell)}(n)} + w^{(\ell)}_n, \quad n = 1, \dots, M,\, \ell = 1, \dots, L.
  \end{align*}
  Here $L$ is the number of observations, $M \ge N$ is the time horizon we allow the prototype to be ``spread out'' over, $h^{(\ell)}: [1,\ldots ,M] \to [1,\ldots ,N]$ are unknown non-decreasing, surjective time changes and $w^{(\ell)}_n$ are independent and identically distributed (iid) random walks.
  The goal is to recover $x$ (up to time warping).

  The currently used method \cite{Bigot2013,Jane1991,Kurtek2011}, consists in first trying to align the different samples,
  i.e., to estimate the time-changes $h^{(\ell)}$, and to average afterwards.
  This seems to work well in regimes where the noise $w^{(\ell)}$ is small (large signal-to-noise ratio),
  but will break down if this is not the case.

  Guided by invariant methods in cryo-EM \cite{Bandeira2017} we then propose the following procedure.
  \begin{enumerate}

    \item Calculate features of $y^{(\ell)}$ that \emph{do not see time warpings}.

    \item Average those features over the independent samples, giving the law of large numbers a chance to cancel out the noise and getting an approximation of the features of $x$.

    \item Invert the averaged features to arrive at a candidate for $x$.

  \end{enumerate}
  Our approach to Step (1) is new and will be presented in this paper. Step (2) and Step (3) will be addressed in future work.
\end{example}

\medskip

A moment's thought reveals that \emph{iterated-sums} of the increments of $x$ are invariant in the desired sense. For example, the simple sum $\sum_i (x_i - x_{i-1})$ or the more complex expressions
\begin{align}
  \label{eq:iteratedSumsExamples}
  \begin{split}
  \sum_i (x_i - x_{i-1})^2, \quad
  \sum_{i_1 < i_2} (x_{i_1} - x_{i_1-1}) (x_{i_2} - x_{i_2-1}), \quad
  \sum_{i_1 \le i_2} (x_{i_1} - x_{i_1-1}) (x_{i_2} - x_{i_2-1}),
  \end{split}
\end{align}
are features of the time series that do not change when warping time, i.e., when repetitions of points, $x_i=x_{i+1}=\cdots=x_{i+j}$ occur in $x$. %
\begin{remark}
  To accommodate repetition of points, here we have conveniently written the sum over an unspecified set of time-points. We can think of the sum taken over $\NOne$, with $x$ being extended constantly as $x_N$ after time $N$.
\end{remark}
However, two questions immediately emerge
\begin{enumerate}
  \item[(A)]
    The three expressions in \eqref{eq:iteratedSumsExamples} are already linearly dependent (adding the first and second sum gives the third). How to store only linearly independent expressions?
  \item[(B)]
    Do iterated-sums of increments give all (polynomial) time warping invariants?
\end{enumerate}

Regarding the first item, it turns out that the above iterated-sums expressions are reminiscent of quasisymmetric functions \cite{MalReu1995}. Consider the space $\R\langle Y_1,Y_2,Y_3,\dots \rangle$ of formal power series in ordered commuting variables $Y_1, Y_2, Y_3, \dots$. By definition, a power series (of finite degree) $Q \in \R\langle Y_1,Y_2,Y_3,\dots \rangle$ is a \textbf{quasisymmetric function} if for all $n \ge 1$, all $i_1 < \cdots < i_n$, all $j_1 < \cdots < j_n$ and all $\alpha_1, \dots, \alpha_n \ge 1$, the coefficient of the monomial $(Y_{i_1})^{\alpha_1} \cdots (Y_{i_n})^{\alpha_n}$ in $Q$ is equal to the one of $(Y_{j_1})^{\alpha_1} \cdots (Y_{j_n})^{\alpha_n}$. First examples are
\begin{align*}
	\sum_i Y_i, \qquad\;\;
  \sum_i (Y_i)^2, \quad
  	\sum_{i_1 < i_2} Y_{i_1} Y_{i_2}, \quad
  	\sum_{i_1 \le i_2} Y_{i_1} Y_{i_2},
\end{align*}
and we see that the invariants given above follow from the evaluation of these quasisymmetric functions at $Y_1 \mapsto x_1 - x_0, Y_2 \mapsto x_2 - x_1, \dots, Y_N \mapsto x_N - x_{N-1}$, and $Y_i \mapsto 0$ for $i \ge N+1$.

Different linear basis for quasisymmetric functions are known. The one of \emph{monomial quasisymmetric functions} of \cite{MalReu1995} is indexed by compositions of integers.
Anticipating the multidimensional case, we write a composition $c_1 + \cdots + c_k = n$ as $[\letter1^{c_1}] \cdots [\letter1^{c_k}]$,
and obtain the correspondence
\begin{align*}
	[\letter1^{c_1}] \cdots  [\letter1^{c_k}]
	\longleftrightarrow
	M_{(c_1, \ldots, c_k)}:=
	\sum_{i_1 < \cdots < i_k} (Y_{i_1})^{c_1} \cdots (Y_{i_k})^{c_k}.
\end{align*}

Quasisymmetric functions are a refinement of symmetric functions and form a commutative unital algebra. The product is just the polynomial product in the power series representation. It amounts to a so-called quasi-shuffle product (see Section \ref{sec:QSh}) in the representation as compositions.
For example, the abstract quasi-shuffle product
\begin{align*}
  [\letter1]  *  [\letter1^3][\letter1^7]
  &=
  [\letter1] [\letter1^3][\letter1^7]
  +
  [\letter1^3][\letter1][\letter1^7]
  +
  [\letter1^3][\letter1^7][\letter1]
  +
  [\letter1^4][\letter1^7]
  +
  [\letter1^3][\letter1^8]
\end{align*}
corresponds to the concrete product of power series
\begin{eqnarray*}
  \lefteqn{\left(\sum_i Y_i\right) \cdot \left(\sum_{i_1<i_2} (Y_{i_1})^3 (Y_{i_2})^7 \right)
  =
  \sum_{i_1 < i_2 < i_3} Y_{i_1} (Y_{i_2})^3 (Y_{i_3})^7
  +
  \sum_{i_1 < i_2 < i_3}  (Y_{i_1})^3 Y_{i_2} (Y_{i_3})^7}\\
  &&\hspace{2cm} +
  \sum_{i_1 < i_2 < i_3} (Y_{i_1})^3 (Y_{i_2})^7Y_{i_3}
  +
  \sum_{i_1 < i_2} (Y_{i_1})^4 (Y_{i_2})^7
  +
  \sum_{i_1 < i_2} (Y_{i_1})^3 (Y_{i_2})^8.
\end{eqnarray*}
The latter equality follows by case distinction for sums over the three indexing variables, which amounts to a summation-by-parts formula. The last two terms in the above product reflect the fact that multiplying sums requires the inclusion of sums over diagonal terms.

It is natural to store the iterated-sums invariants of the discrete time series $x$ as a linear map $\DS(x)$ on the quasi-shuffle algebra of compositions, by defining the pairing
\begin{align*}
	\langle [\letter1^{c_1}] \cdots [\letter1^{c_k}], \DS(x) \rangle
  	:=
  	\sum_{i_1 < \cdots < i_k} (\Delta x_{i_1})^{c_1} \cdots (\Delta x_{i_k})^{c_k}.
\end{align*}
Here $\Delta x_i := x_i - x_{i-1}$ for $1 \le i \le N$, and as above we extend $x$ constantly, so that $\Delta x_i := 0$ for $i \ge N+1$.
From the correspondence between the product of power series and the quasi-shuffle product of compositions mentioned above we deduce that
\begin{align*}
  \langle [\letter1^{p_1}] \cdots [\letter1^{p_k}], \DS(x) \rangle
  \cdot
  \langle [\letter1^{q_1}] \cdots [\letter1^{q_{l}}], \DS(x) \rangle
  &=
  \langle [\letter1^{p_1}] \cdots [\letter1^{p_k}] * [\letter1^{q_1}] \cdots [\letter1^{q_{l}}], \DS(x) \rangle.
\end{align*}
Hence, $\DS(x)$, which we call iterated-sums signature, is an algebra morphism (from the quasi-shuffle algebra to the underlying base field $\field$). Since compositions form a linear basis, this answers Question (A) above -- in the case $d=1$. We will come back to Question (B) in Section \ref{sec:inv}.

The commutative algebra of quasisymmetric functions is the free quasi-shuffle algebra over one generator and it is - as we just saw - the correct framework to store iterated-sums for a one-dimensional time-series. The appropriate generalisation of this algebra to arbitrary dimension $d \ge 1$, that is, the free quasi-shuffle algebra over $d$ generators, was carried out by Hoffman \cite{Hoffman2000}.

The aforementioned amounts to saying that iterated-sums signature $\DS(x)$ is an element of the dual space of the quasi-shuffle algebra over $d$ generators. It can therefore be represented as an infinite word series with iterated-sums of the time series $x$ as coefficients. Its compatibility with the quasi-shuffle product together with the fact that the latter can be seen as a deformation of the classical shuffle product \cite{FoiPat2016} suggests to consider $\DS(x)$ as a discrete analog of Chen's \emph{iterated-integrals signature} over continuous curves \cite{Chen1957,Ree1958}. The latter plays an important role in the theory of controlled ordinary differential equations (ODEs), stochastic analysis and Lyons' theory of rough paths \cite{FrizVic2010,Lyons1998}. Such a large spectrum of applications reflects the important property of iterated-integrals to provide - in some sense - a  complete representation of a curve, so that arbitrary functionals on curves should be well approximated by functions on its signature.
There is a caveat though. Iterated integrals are tailor made to approximate functionals that stem from controlled ODEs. But as is quickly realised, this does \emph{not} mean that the iterated-integrals signature is an optimal representation for other input-output systems.
For example, since a controlled ODE - and hence also the signature - cannot see tree-like excursions, \emph{the iterated-integrals signature of a one-dimensional path reveals nothing about the path, except for its increment}.%
There are several procedures to circumvent this shortcoming, and to obtain information even about tree-like parts of a curve using signature.
These procedures usually consists of lifting the path to a higher-dimensional curve
and calculating the signature of it.
The aforementioned limitations of the iterated-integrals signature with respect to tree-like paths prompts us to propose instead the use of ``discrete time signature'' $\DS(x)$, which, instead of storing iterated-integrals, gathers iterated-sums.
\begin{remark}
  For the precise definition of ``tree-like'' see \cite{LyonsHambly2005}; but one can think of a curve that completely ``tracks back''. In particular in dimension $1$, every curve that has coinciding start- and endpoint is tree-like.
\end{remark}

The paper is organised as follows. Section \ref{sec:QSh} recalls the notion of quasi-shuffle Hopf algebra and quasisymmetric functions. In Section \ref{sec:dsig} we introduce the iterated-sums signature and show its character property with respect to the quasi-shuffle Hopf algebra. Moreover, we show that Chen's property is satisfied, but that Chow's Theorem does not hold.
Hence, while mirroring the setup of Chen's iterated-integrals signature to some extent,
interesting differences emerge. It turns out that our description of the iterated-sums signature
is nicely related to the work \cite{NovThi2004} on ``multidimensional'' generalisation of quasisymmetric functions,
and we dwell on this briefly in Remark \ref{rmk:multidimQsym}.
In Section \ref{sec:inv} we show the iterated-sums signature contains (almost) all time warping invariants.
In Section \ref{sec:HoffSig} we use a specific Hopf algebra isomorphism, known as Hoffman's exponential,
to relate the iterated-sums signature to Chen's iterated-integrals signature (of an infinite-dimensional path). This includes in particular relating the continuous and discrete area operations.


In the following all algebraic structures are defined over a base field $\field$ of characteristic zero.
The reader is invited to think of the field $\field$ as the reals, $\field=\R$, or the complex numbers, $\field=\mathbb C$, throughout.

We denote $\NZero\coloneqq\{0,1,2,\dots\}$ and $\NOne\coloneqq\{1,2,\dots\}$. All (co)algebras are (co)unital and (co)asso\-ci\-a\-tive unless otherwise stated. For details on Hopf algebras the reader is referred to \cite{Cartier2007,Hazewinkel2010,Luoto2013,Manchon2008,Reutenauer1993}.


\section{Quasi-shuffle Hopf algebra}
\label{sec:QSh}

The notion of quasi-shuffle product appeared first in a 1972 article by Cartier \cite{Cartier1972}. Its Hopf algebraic relevance was explored in the 1979 paper~\cite{NewRad1979}. Two decades later, Hoffman \cite{Hoffman2000} provided a comprehensive account of the quasi-shuffle product in a Hopf algebraic framework. Meanwhile, quasi-shuffle products appeared under different names, i.e., modified shuffle product \cite{Gaines1994,LiLiu1997}, sticky-shuffle \cite{Hudson2009,Hudson2012}, overlapping shuffle \cite{Hazewinkel2001}, stuffle and harmonic product \cite{Zudilin2003}.

We recall the inductive definition of the quasi-shuffle product following
Hoffman \cite{Hoffman2000}. See also \cite{BCE2018,EfMPW1}. Our
starting point is the alphabet $A=\{\letter 1,\letter 2,\dotsc,\letter d\}$,
which we augment to a free \emph{commutative} semigroup, $\mathfrak A$, by
defining a commutative product denoted by square brackets, $[- -]\colon\mathfrak A
\times \mathfrak A \to \mathfrak A$. For example, the product between the
letters $\letter 1,\letter 2\in A$ is written $[\letter 1 \letter 2]=[\letter 2
\letter 1]$. Any iteration of the product in
$\mathfrak A$ can be simplyfied to an expression containing a single pair of
brackets, that is, $[\letter{i_1} \cdots  \letter{i_{n}}]
\coloneqq[\letter{i_1} [\cdots [\letter{i_{n-1}} \letter{i_{n}}]]\cdots]$. For
instance, $[\letter 1\letter 2\letter 3]=[\letter 1[\letter 2\letter 3]]$ in
$\mathfrak A$. Elements in the tensor algebra $T(\mathfrak A)$ over
(the vector space spanned by) $\mathfrak A$ are denoted by words, i.e., we denote the
tensor product by \emph{concatenation}, or juxtaposition of basis
elements. The neutral element for this product is the empty word, denoted by
$e$. The augmentation ideal is defined by $T_+(\mathfrak
A)\coloneqq\oplus_{n>0} {\mathfrak A}^{\otimes n}$ such that $T(\mathfrak A)= \field e
\oplus T_+(\mathfrak A)$.

The commutative \textbf{quasi-shuffle} product $m_{\star }\colon T(\mathfrak A)\otimes T(\mathfrak A)\to T(\mathfrak A)$, $u \star v:=m_\star(u,v)$, is introduced by inductively defining $e\star u\coloneqq u \eqqcolon u\star e$, for all $u\in T(\mathfrak A)$, and
\begin{equation}
\label{quasishuffle}
	ua \star vb\coloneqq (u\star vb)a+(ua\star v)b+(u\star v)[ab],
\end{equation}
for $u,v\in T(\mathfrak A)$ and $a,b \in \mathfrak A$. For example, $\letter 2 \star \letter 3=\letter 2\letter 3 +\letter 3\letter 2+[\letter 2 \letter 3]$ and
\begin{equation}
\label{ex:qsh}
	\letter 3 \star  \letter 4 [\letter{12}]
	= \letter{34}[\letter{12}]
	+ \letter{43}[\letter {12}]
	+\letter 4 [\letter{12}] \letter 3
 	+ [\letter{34}][\letter{12}]
 	+ \letter 4[\letter{123}].
\end{equation}
The tensor algebra is naturally graded by the length of words, $\ell(w_1\dotsm w_n)=n$ for $w_1\dotsm w_n \in T(\mathfrak A)$. However, in light of the new product \eqref{quasishuffle}, which is not homogenous with respect to the number of letters, we introduce the weight grading on $T(\mathfrak A)$, denoted $|\cdot|$,  by declaring that $|e|=0$, $|\letter a|=1$ for all $\letter a\in A$ and $|[pq]|=|p|+|q|$ for all $p,q\in \mathfrak A$. Finally, for a word $w=w_1\dotsm w_n\in T(\mathfrak A)$ we define its weight to be $|w|=|w_1|+\dotsm+|w_n|$.

Let $\delta \colon T(\mathfrak A)\to T(\mathfrak A)\otimes T(\mathfrak A)$ denote the \textbf{deconcatenation coproduct} defined on a nonempty word $w=w_1\cdots w_n \in T(\mathfrak A)$ by
\begin{equation}
\label{decon}
	\delta(w)\coloneqq
	w \otimes e + e \otimes w
	+ \sum\limits_{i=1}^{n-1} w_1\cdots w_i\otimes w_{i+1}\cdots w_n,
\end{equation}
and $\delta(e)=e \otimes e$.
It turns $T(\mathfrak A)$ into a connected graded coalgebra, for both the length and weight grading. For any word $w \in T_+(\mathfrak A)$ the reduced coproduct is defined by $\delta'(w)\coloneqq\delta(w)- w \otimes e - e \otimes w$. ``Sweedler's notation'' will be employed for both coproducts:  $\delta(w)\eqqcolon\sum_{(w)} w_{(1)} \otimes w_{(2)}$ and $\delta'(w)\eqqcolon\sideset{}{'}\sum_{(w)} w_{(1)} \otimes w_{(2)}$. The canonical counit map $\varepsilon\colon T(\mathfrak A) \to \field$ is defined to be $\varepsilon(\lambda e)=\lambda \in \field$ and zero on $T_+(\mathfrak A)$. In \cite{Hoffman2000} Hoffman showed the following

\begin{theorem}[Quasi-shuffle Hopf algebra]
  \label{theorem:quasiShuffleHopfAlgebra}
  1. $H_{\mathrm{qsh}}=(T(\mathfrak A),\star,\delta,\varepsilon ,|\cdot|)$ is a graded, connected, commutative, non-cocommutative Hopf algebra.

  2. The antipode $\alpha \colon H_{\mathrm{qsh}} \to H_{\mathrm{qsh}}$ is given by
    \begin{equation}
    \label{antipodecomp}
       \alpha(w_1 \dotsm w_n) = (-1)^n \sum_{I\in\mathcal C(n)}I[w_n \dotsm w_1].
    \end{equation}
    Here $\mathcal C(n)$ is the set of all compositions of the integer $n$, i.e., tuples $(i_1,\dotsc, i_p)$ of positive integers such that $i_1+\dotsb+i_p=n$.
    Given $I=(i_1,\dotsc,i_p) \in \mathcal C(n)$ and a word $w=w_1\dotsm w_n
    \in T(\mathfrak A)$ of length $\ell(w)=n>0$, we define a new word
    $I[w]\in T(\mathfrak A)$ by \begin{equation*}
      I[w]\coloneqq[w_1\dotsm w_{i_1}][w_{i_1+1} \dotsm w_{i_1+i_2}]\dotsm
    [w_{i_1+\dotsb+i_{p-1}+1}\dotsm w_n].  \end{equation*} Here (as well as later) we are using the suitable 
    convention that $[a]:=a$ for all $a \in \mathfrak A$.
\end{theorem}

\begin{remark}[Shuffle Hopf algebra]
\label{rmk:shuffle}
If the semigroup $\mathfrak A$ is trivial, i.e., if $[\letter i \letter j]=0$ for any letters $\letter i,\letter j \in A$, then the quasi-shuffle product  \eqref{quasishuffle} reduces to Chen's commutative shuffle product on $T(A)$:
$$
	v \letter i \shuffle w \letter j
	\coloneqq (v \shuffle w \letter j)\letter i + (v\letter i \shuffle w)\letter j,
$$
for $v,w \in T(A)$ and $\letter i , \letter j \in A$. Observe that in this case $|w|=\ell(w)$ for any word and $H_{\shuffle}=(T(\mathfrak A),\shuffle,\delta,\varepsilon ,\ell)$ is the classical shuffle Hopf algebra over the alphabet $A$. From \eqref{antipodecomp} it follows that the antipode on $H_{\shuffle}$ is given by $\alpha(\letter i_1 \dotsm \letter i_n)=(-1)^n \letter i_n \dotsm \letter i_1$. See \cite{Reutenauer1993} for a comprehensive account on $H_{\shuffle}$.
\end{remark}

\begin{remark}[A remark on dimensions.]
There is a simple way of computing the Hilbert series
\[
  	G(t)\coloneqq\sum_{n\ge0}t^n\dim T(\mathfrak A)_n
\]
of $T(\mathfrak A)$, where $T(\mathfrak A)_n\coloneqq\mathbb F\{w:|w|=n\}$ is the homogeneous (for the weight grading) component of degree $n$ of the quasi-shuffle algebra. It is not hard to see that all such words are of the form $I[\letter a_1\dotsm\letter a_n]$ for some composition $I=(i_1,\dotsc,i_p)\in\mathcal C(n)$ and letters $\letter a_1,\dotsc,\letter a_n\in A$, in the notation of Theorem \ref{theorem:quasiShuffleHopfAlgebra}. In total, in each block of size $i_1, i_2,\dotsc, i_p$ we are allowed to put a symmetric monomial of length $i_j$ of which there are exactly $\binom{d-1+i_j}{i_j}$ -- this is the dimension of the degree-$i_j$ part of the symmetric algebra $S(A)$. Therefore
\[
	\dim T(\mathfrak A)_n
	=\sum_{(i_1,\dotsc,i_p)\in\mathcal C(n)}\binom{d-1+i_1}{i_1}\dotsm\binom{d-1+i_p}{i_p}.
\]
A simple computation shows that in fact
\[
	\binom{d-1+i}{i}=\frac{d(d+1)\dotsm(d+i-1)}{i!}=\frac{1}{i!}(d)_i,
\]
where the Pochhammer symbol (or rising factorial) appears on the righthand side. It is well known that their exponential generating function equals the hypergeometric function
\[
	{}_1F_0(d;t)=1+\sum_{i=1}^\infty(d)_i\frac{t^i}{i!}=(1-t)^{-d}.
\]
Therefore
\begin{align*}
	G(t)=\sum_{n=0}^\infty t^n\sum_{(i_1,\dotsc,i_p)\in\mathcal C(n)}
		\frac{(d)_{i_1}\dotsm(d)_{i_p}}{i_1!\dotsm i_p!}
	&=1+\sum_{p=1}^\infty\left( \sum_{i=1}^\infty (d)_i\frac{t^i}{i!} \right)^p\\
      	&=\sum_{p=0}^\infty\left( (1-t)^{-d}-1 \right)^p=\frac{(1-t)^d}{2(1-t)^d-1}.
\end{align*}
The coeffficients of these Hilbert series can be found in column $d$ of the OEIS sequence \verb|A261780|.
\end{remark}

Define the scalar product $\langle - , -\rangle \colon T(\mathfrak A) \otimes T(\mathfrak A) \to \field$ for any words $u,v \in T(\mathfrak A)$ by $\langle u , v\rangle \coloneqq 1$ if $u=v$ and zero else. It permits to identify the graded dual of $H_{\mathrm{qsh}}$ as word series, i.e., $c=\sum_{w \in  T(\mathfrak A)} \langle w , c\rangle w \in T\dlrp{\mathfrak A}\eqqcolon H_{\mathrm{qsh}}^*$, which is a non-commutative (topological) Hopf algebra with concatenation
as convolution product, denoted by $m_\centerdot \colon H_{\mathrm{qsh}}^* \otimes H_{\mathrm{qsh}}^* \to H_{\mathrm{qsh}}^*$, $c \centerdot c'\coloneqq  m_\centerdot (c \otimes c')$, and de-quasi-shuffling as coproduct \cite{Hoffman2000}. In more concrete terms, this means that given two such series $c,c'\in T\dlrp{\mathfrak A}$ their convolution product $c \centerdot c'\coloneqq  m_\mathbb{F}(c \otimes c')\delta : H_{\mathrm{qsh}} \to \field$ may be written as
\[
	c \centerdot c'
	=\sum_{w\in T(\mathfrak A)}\sum_{uv=w}\langle u,c\rangle\langle v,c'\rangle w
	=\sum_{w\in T(\mathfrak A)}\langle \delta(w), c \otimes c'\rangle w.
\]
Of particular interest are characters, i.e., algebra morphisms $c \in H_{\mathrm{qsh}}^*$. They satisfy $\langle e,c\rangle=1$ and $\langle u \star v,c\rangle=\langle u,c\rangle\langle v,c\rangle$, for $u,v \in  H_{\mathrm{qsh}}$. The first property requires that the coefficient $\langle e , c\rangle=1$ and the second is equivalent to $c$ being group-like in $H_{\mathrm{qsh}}^*$, which means that for $u,v \in H_{\mathrm{qsh}}$
$$
		\langle u \star v , c\rangle
		=\langle u \otimes v , \Delta_{\mathrm{qsh}}(c)\rangle
		=\langle u \otimes v, c \otimes c\rangle,
$$
where the de-quasi-shuffling coproduct is defined on words by
$$
	\Delta_{\mathrm{qsh}}(w):=\sum_{u,v \in T(\mathfrak A)} \langle u \star v , w\rangle u \otimes v.
$$
The set of characters, denoted by $\mathscr{G}$, forms a group with the inverse $c^{-1}=c\circ\alpha$. The corresponding Lie algebra, $\mathfrak{g} \subset H_{\mathrm{qsh}}^*$, consists of so-called infinitesimal characters, which map the empty word and any non-trivial product in $H_{\mathrm{qsh}}$ to zero. One can define the exponential map as a power series with respect to the convolution product which maps $\mathfrak{g}$ bijectively to $\mathscr{G}$, i.e., $\exp^\centerdot(f) := \varepsilon + \sum_{j > 0} \frac{1}{j!}f^{\centerdot j} \in \mathscr{G}$. Because $T(\mathfrak A)$ is a graded connected Hopf algebra, this expression becomes a finite sum when evaluated on homogeneous elements of $T(\mathfrak A)$, so we do not have to deal with convergence issues. Its inverse is the logarithm, $\log^\centerdot(\varepsilon + (c-\varepsilon)) =\sum_{i \ge 1} \tfrac{(-1)^{i-1}}{i}(c-\varepsilon)^{\centerdot i} \in \mathfrak{g}$. Again, the sum applied to any word $w \in T(\mathfrak A)$ terminates after $|w|$ terms, as $(c-\varepsilon)(e) =0$.

\begin{notation}
\label{goodnotation}
We introduce a particular notation for words in $T(\mathfrak A)$, which will be
useful in the sequel. The convention to identify $[a]\coloneqq a$, for $a \in
\mathfrak A$, permits to write any word in $T(\mathfrak A)$ as a concatenation
of brackets, i.e., $w = [u_1] \cdots [u_k] \in T(\mathfrak A)$, for $u_1,
\ldots, u_k \in \mathfrak A$.
\end{notation}

We come back to the setting of the introductory section with only a single letter,  $A=\{\letter 1\}$. Then, in each degree $n$, $T(\mathfrak A)$ has a single word of length one, $[1^n] \in \mathfrak A$,
and any basis element (or word) is of the form $w=[\letter 1^{k_1}][\letter
1^{k_2}]\dotsm[\letter 1^{k_n}]$ for some integers $k_1,\dotsc,k_n>0$. It is
easy to see that then the tuple $(k_1,\dotsc,k_n)$ is a composition of the
integer $|w|$ of length $n=\ell(w)$. In \cite{Hoffman2000} Hoffman describes a
unital algebra isomorphism $\Sigma$ between the quasi-shuffle algebra
$H_{\mathrm{qsh}}$, for $A=\{\letter 1\}$, and the algebra $\mathrm{QSym}$ of
quasisymmetric functions in the ordered set of commuting variables $\{Y_i\}_{i
\in \mathbb{N}_+}$ \cite{Gessel1984}, defined by taking a word in $T(\mathfrak
A)$ to an iterated sum
\begin{equation}
\label{Qsym}
  \Sigma\left([\letter 1^{k_1}][\letter 1^{k_2}]\dotsm[\letter 1^{k_n}]\right)
  \coloneqq \sum_{1 \le i_1 < \cdots < i_n} (Y_{i_1})^{k_1} \cdots (Y_{i_n})^{k_n}
  \eqqcolon M_{(k_1,\ldots,k_n)}.
\end{equation}
Here $\Sigma(e)=M_{0}=1$. Then, the correspondence of the introduction is explicitly given by
\begin{align*}
 \lefteqn{ \Sigma\left( [\letter1] \right)\cdot \Sigma\left( [\letter1^3][\letter1^7] \right)
  =
  \left(\sum_i Y_i\right) \cdot \left(\sum_{i_1<i_2} (Y_{i_1})^3 (Y_{i_2})^7 \right)} \\
  &=
  \sum_{i_1 < i_2 < i_3} Y_{i_1} (Y_{i_2})^3 (Y_{i_3})^7
  +
  \sum_{i_1 < i_2 < i_3} (Y_{i_1})^3 Y_{i_2} (Y_{i_3})^7
  +
  \sum_{i_1 < i_2 < i_3}  (Y_{i_1})^3 (Y_{i_2})^7 Y_{i_3}\\
  &\qquad
  +
  \sum_{i_1 < i_2} (Y_{i_1})^4 (Y_{i_2})^7
  +
  \sum_{i_1 < i_2} (Y_{i_1})^3 (Y_{i_2})^8 \\
  &=
  \Sigma\left( [\letter1] [\letter1^3][\letter1^7] + [\letter1^3][\letter1][\letter1^7] + [\letter1^3][\letter1^7][\letter1] + [\letter1^4][\letter1^7] + [\letter1^3][\letter1^8] \right) \\
  &=\
  \Sigma\left( [\letter1] \star [\letter1^3][\letter1^7] \right),
\end{align*}
where the second equality is an example of summation-by-parts for products of iterated sums. 

The $M_{(k_1,\ldots,k_n)}$ of \eqref{Qsym} are the monomial quasisymmetric functions, which form a basis for $\mathrm{QSym}$. The Hopf algebra $\mathrm{QSym}$ is a generalisation of the classical Hopf algebra $\mathrm{Sym}$ of symmetric functions. It was defined and studied by Gessel \cite{Gessel1984}, based on earlier work by Stanley, and plays a rather distinguished role in modern algebraic combinatorics, with ramifications into several other fields of mathematics. Its graded dual is known as the connected graded cocommutative Hopf algebra $\mathrm{NSym}$ of noncommutative symmetric functions. The iterated-sums signature corresponding to a one dimensional discrete time series, alluded to in the first section, is an element in $\mathrm{NSym}$. Further below, in Section \ref{sec:dsig}, we consider the multidimensional generalisation of quasisymmetric functions (of level $d$ in the terminology of \cite{NovThi2004}) and its corresponding iterated-sums signature. We close this section by mentioning that Malvenuto's and Reutenauer's Hopf algebra of permutations \cite{MalReu1995} plays an important part in the understanding of the relation between the objects $\mathrm{Sym}$, $\mathrm{QSym}$ and $\mathrm{NSym}$. The interested reader is referred to \cite{AS2005,ABS2006} and to \cite{Luoto2013} for a readable introduction, including a brief historical overview.


\subsection{Half-shuffles}
\label{ssec:halfshuf}

Aiming at understanding the discrete analog of the ${\mathrm{area}}$ operation (to be introduced further below), we take a more refined approach at the quasi-shuffle product by observing that $m_\star$ may be split into three products, i.e., left and right \emph{half-shuffles} and a third product
\begin{equation}
\label{ex:halfshuf}
	ua\qsucc vb\coloneqq (ua\star v)b,
	\enspace
	ua\qprec vb\coloneqq(u\star vb)a,
	\enspace
	ua\qdot vb\coloneqq(u\star v)[ab],
\end{equation}
so that $u\star v=u\qprec v+u\qsucc v+u\qdot v$. For instance (c.f.~Example \ref{ex:qsh})
\begin{equation}
\label{ex:qshdend}
\begin{split}
  \letter 3\qsucc\letter 4[\letter{12}]
  &=\letter{34}[\letter{12}]+\letter{43}[\letter{12}]+[\letter{34}][\letter{12}]\\
  \letter 3\qprec\letter 4[\letter{12}]
  &=\letter 4[\letter{12}]\letter 3\\
  \letter 3\qdot\letter 4[\letter{12}]
  &=\letter 4[\letter{123}].
\end{split}
\end{equation}

Noticing the particular relation $ua \qsucc vb = vb \qprec ua$ which is equivalent to $m_\star$ being commutative, it is not hard to show that the quasi-shuffle algebra  $H_{\mathrm{qsh}}=(T(\mathfrak A),\star )$  becomes a commutative tridendriform algebra, $(T(\mathfrak A),{\qprec},{\qsucc},{\qdot})$, as defined by Loday and Ronco \cite{LodRon2004}.

\begin{remark}
A similar splitting holds for the shuffle algebra in Remark \ref{rmk:shuffle}. We can write the shuffle product $m_\shuffle$ on $T(A)$ as a sum of the two half-shuffles
\[
  u\letter a\succ v\letter b\coloneqq(u\letter a\shuffle v)\letter b,
  \enspace
  u\letter a\prec v\letter b\coloneqq(u\shuffle v\letter b)\letter a,
\]
so that $u\letter a\shuffle v\letter b=u\letter a\prec v\letter b+u\letter a\succ v\letter b$. Again, we check quickly that the commutativity of the shuffle product is equivalent to $u\letter a\succ v\letter b = v \letter b\prec u\letter a$. In fact, the triple $(T(A),{\prec},{\succ})$ is also known as a commutative dendriform or Zinbiel algebra.
\label{rmk:shdend}
\end{remark}


\subsection{Hoffman's exponential}
\label{ssec:Hoffman}

Shuffle and quasi-shuffle Hopf algebras are more tightly related than Remark \ref{rmk:shuffle} may adumbrate. Indeed, Hoffman proved in \cite{Hoffman2000} that $H_{\shuffle}=(T(\mathfrak A),\shuffle,\delta)$ and $H_{\mathrm{qsh}}=(T(\mathfrak A),\star,\delta)$ are isomorphic as Hopf algebras. We briefly recall this result. Let $T(\mathfrak A)$ be equipped with the commutative shuffle product $m_\shuffle \colon T(\mathfrak A)\otimes T(\mathfrak A)\to T(\mathfrak A)$ inductively defined by $u[a] \shuffle v[b] \coloneqq (u\shuffle v[b])[a]+(u[a]\shuffle v)[b]$, for $u,v\in T(\mathfrak A)$ and $a,b\in \mathfrak A$. The empty word, $e$, is the unit for this product. Recall the notation $I[w]$ introduced in Theorem \ref{theorem:quasiShuffleHopfAlgebra}.

\begin{theorem}[Hoffman's isomorphism]{\cite{Hoffman2000}}\label{thm:HoffExp}
  There exists a Hopf algebra isomorphism $\Phi_{\mathrm{H}} \colon(T(\mathfrak A),\shuffle,\delta) \to(T(\mathfrak A),\star,\delta)$, given explicitly by the so-called \emph{Hoffman exponential}
\begin{equation}
  \label{HoffIso1}
      \Phi_{\mathrm{H}}(w)\coloneqq\sum_{(i_1,\dotsc,i_p)\in\mathcal C(\ell(w))}
      \frac{1}{i_1!\dotsm i_p!}I[w] .
    \end{equation}
    Its inverse also admits an explicit expression, namely the Hoffman logarithm
    \begin{equation}
  \label{HoffIso2}
  \Phi_{\mathrm{H}}^{-1}(w)\coloneqq\sum_{(i_1,\dotsc,i_p)\in\mathcal C(\ell(w))}
  \frac{(-1)^{\ell(w)-p}}{i_1\dotsm i_p}I[w].
\end{equation}
\end{theorem}

Some examples: $\Phi_{\mathrm{H}}([\letter i])=[\letter i]$ and for the words $[\letter 1][\letter 2] \in T(\mathfrak A)$ and $[\letter 1][\letter{23}][\letter 4] \in T(\mathfrak A)$ we find
\allowdisplaybreaks
\begin{align*}
	\Phi_{\mathrm{H}}([\letter 1][\letter 2])
	&=[\letter 1][\letter 2]+\frac12[\letter{12}]\\
      	\Phi_{\mathrm{H}}([\letter 1][\letter{23}][\letter 4])
    	&=[\letter 1][\letter{23}][\letter 4]
          +\frac12[\letter{123}][\letter 4]
      	+\frac12[\letter 1][\letter{234}]
      	+\frac16[\letter{1234}]\\
  	\Phi_{\mathrm{H}}([\letter 1] \shuffle [\letter 2])
    	   = \Phi_{\mathrm{H}}([\letter 1][\letter 2]+[\letter 2][\letter 1])
    	& = [\letter 1][\letter 2]
        	+ \frac{1}{2}[\letter 1 \letter 2]
        	+[\letter 2][\letter 1]
        	+ \frac{1}{2}[\letter 2 \letter 1]\\
     	&=[\letter 1][\letter 2] + [\letter 2][\letter 1] + [\letter 1 \letter 2]
     	  =\Phi_{\mathrm{H}}([\letter 1]) \star \Phi_{\mathrm{H}}([\letter 2]).
\end{align*}
In the second example, the terms correspond to the compositions $(1,1,1)$, $(2,1)$, $(1,2)$ and $(3)$ of the integer $3$, in that order. Recall that the particular Notation \ref{goodnotation} for words $w = [u_1] \cdots [u_k] \in T(\mathfrak A)$, for $u_1, \ldots, u_k \in \mathfrak A$, is in place. Also, note that the number of letters in each of the terms corresponds to the length of the composition. The reader is referred to \cite{Hoffman2000,HofIha2017} for more details. See also \cite{EfMPW1} for an application in stochastic analysis.

In Section \ref{sec:HoffSig} we will show that $\Phi_{\mathrm H}$ is nicely compatible with comparing
the iterated-sums signature on one side with the iterated-integrals signature on the other.
The following two lemmas are going to be used in \Cref{ssec:area}, where we address the area operation in the context of the iterated-sums signature.

\begin{lemma}
\label{lem:Hoffsplit}
The image of any nonempty word $w=w_1\dotsm w_n \in T(\mathfrak A)$ under Hoffman's isomorphism can be split into two parts as follows:
\begin{equation}
\label{eq:Hoffsplit}
	\Phi_{\mathrm{H}}(w)=\Phi_{\mathrm{H}}(w_1\dotsm w_{n-1})w_n+R_{\mathrm{H}}(w),
\end{equation}
where the remainder term
\[
	R_{\mathrm{H}}(w)
	=\sum_{\substack{I=(i_1,\dotsc,i_p)\in\mathcal C(\ell(w))\\i_p>1}}\frac{1}{i_1!\dotsm i_p!}I[w].
\]
\end{lemma}

The verification of the lemma is left to the reader. This splitting of Hoffman's isomorphism implies the following important result.

\begin{lemma}
\label{lmm:HofDzu}
Let $u \in T(\mathfrak A)$ and $a,b\in\mathfrak A$. Then
\begin{equation}
\label{eq:HoffLie}
	\Phi_{\mathrm{H}}\big(u([a][b]-[b][a])\big)
	=\Phi_{\mathrm{H}}(u[a])[b]-\Phi_{\mathrm{H}}(u[b])[a].
\end{equation}
\end{lemma}

\begin{proof}
From Lemma \ref{lem:Hoffsplit} and linearity of $\Phi_{\mathrm{H}}$, we deduce that
\[
	\Phi_{\mathrm{H}}\big(u([a][b]-[b][a])\big)
	=\Phi_{\mathrm{H}}(u[a])[b] -\Phi_{\mathrm{H}}(u[b])[a]
	+ R_{\mathrm{H}}(u[a][b]) - R_{\mathrm{H}}(u[b][a]).
\]
Since the semigroup $\mathfrak A$ is commutative, for any composition $I=(i_1,\dotsc,i_p)\in\mathcal C(n)$ with $i_p\geq 2$ we have that
\begin{align*}
      I\big[u[a][b]\big]
      &=[u_1\dotsm u_{i_1}][u_{i_1+1}\dotsm
      u_{i_1+i_2}]\dotsm[u_{i_1+\dotsb+i_{p-1}+1}\dotsm u_{n-2}ab]\\
      &=[u_1\dotsm u_{i_1}][u_{i_1+1}\dotsm u_{i_1+i_2}]\dotsm
      [u_{i_1+\dotsb+i_{p-1}+1}\dotsm u_{n-2}ba]\\
      &=I\big[u[b][a]\big].
\end{align*}
Therefore, the equality $R_{\mathrm{H}}(u[a][b])=R_{\mathrm{H}}(u[b][a])$ holds, which implies the identity \eqref{eq:HoffLie}.
\end{proof}


\section{Iterated-sums signatures}
\label{sec:dsig}

We consider a discrete time series $x \in (\field^d)^N$ as an element of
\begin{align*}
  (\field^d)^{\NOne}_c
  :=
  \left\{ x: \NOne \to \field^d: \exists\, N \ge 1 \text{ such that }\, x_N = x_n\, \forall n \ge N \right\},
\end{align*}
the space of infinite time series that are eventually constant,
by extending it constantly.
In this section we will see that the appropriate algebraic setting for
iterated-sums, combined into the map $\DS(x)$, is that of a character
on the quasi-shuffle Hopf algebra $H_{\mathrm{qsh}}=(T(\mathfrak A),\star,\delta, \varepsilon, |\cdot|)$ over the semigroup $\mathfrak A$ corresponding to the alphabet $A=\{\letter 1,\letter 2,\dotsc,\letter d\}$, introduced in Section \ref{sec:QSh}.

The following notation for elements in the time series $x$ is put in place:
$$
	x_j=(x_j^{[\letter 1]},\dotsc,x_j^{[\letter d]}) \in \field^d.
$$
Next we define the corresponding time series
$$
	\Delta x=((\Delta x)_1, (\Delta x)_2, \dotsc, (\Delta x)_N)
$$
with increments $(\Delta x)_n \coloneqq x_n - x_{n-1} \in \field^d$, for $n\ge1$, as
entries.
The new notation is extended
to include all brackets in $\mathfrak A$ by defining
\[
	x_j^{[\letter a_1\dotsm\letter a_p]} \coloneqq x_j^{[\letter a_1]}\dotsm x_j^{[\letter a_p]}.
\]

\begin{definition}
  The \textbf{iterated-sums signature} of the time series $x$ is the
  two-parameter family $(\DS(x)_{n,m}\mid 0\le n\leq m\in\NZero )$ of linear
  maps from $T(\mathfrak A)$ to $\field$ such that $\DS(x)_{n,n}=\varepsilon$,
  and defined recursively by $\langle e,\DS(x)_{n,m}\rangle\coloneqq1$, and for
  $a_1\dotsm a_p \in T(\mathfrak A)$
$$
  \langle [a_1] \cdots [a_p],\DS(x)_{n,m}\rangle \coloneqq
  \sum_{j=n+1}^m\langle [a_1]\dotsm [a_{p-1}],\DS(x)_{n,j-1}\rangle\Delta x_j^{[a_p]}.
$$
\end{definition}

Hence, the iterated-sums signature is a word series in $H_{\mathrm{qsh}}^*$
\begin{align}
\label{def:discretesig1}
  \DS(x)_{n,m}
  &= \sum_{[u_1] \cdots [u_k] \in T(\mathfrak A)}
  \langle [u_1] \cdots [u_k],\DS(x)_{n,m}\rangle [u_1] \cdots [u_k]
\end{align}
with iterated sums over increments of $x$ as coefficients, defined as
\begin{align}
\label{def:discretesig2}
  \langle [u_1] \cdots [u_k],\DS(x)_{n,m}\rangle
  &= \sum_{n<i_1<i_2 < \cdots < i_k \le m} \Delta x_{i_1}^{[u_1]} \Delta x_{i_2}^{[u_2]} \cdots \Delta x_{i_k}^{[u_k]}.
\end{align}
For example
\[
  \langle[\letter  1][\letter{12}],\DS(x)_{n,m}\rangle
  = \sum_{n<i_1<i_2\le m}\Delta x_{i_1}^{[\letter 1]}\Delta x_{i_2}^{[\letter 1]}\Delta x_{i_2}^{[\letter 2]}.
\]

We extend this definition to all $n,m\in\NZero$ by setting $\langle w,\DS(x)_{n,m}\rangle=0$ whenever $m<n$.
\begin{remark}
  \label{rmk:vanish}
  An easy consequence of this definition is that the coefficient $\langle w,\DS(x)_{n,m}\rangle$ vanishes whenever $\ell(w)>m-n$.
\end{remark}

The proof of the following lemma is straightforward.

\begin{lemma}
\label{lemma:discdiff}
Let $x=(x_n)_{n\ge0}$ and $x'=(x'_n)_{n\ge0}$ be two time series, and denote by $xx'\coloneq(x_nx'_n)_{n\ge0}$. Then the increment of the product $xx'$ is given by a generalised Leibniz rule
\[
	(\Delta xx')_n =
	 x'_{n-1}(\Delta x)_n
        +x_{n-1}(\Delta x')_n
        +(\Delta x)_n(\Delta x')_n.
\]
\end{lemma}

More importantly, we have the following:

\begin{theorem}
  \label{theorem:quasiShuffleChen}
  \begin{enumerate}
   \item \textbf{(Quasi-shuffle identity)}
       For each $n\le m$, the map $\DS(x)_{n,m}:H_{\mathrm{qsh}} \to \field$ is a quasi-shuffle Hopf algebra character.

   \item \textbf{(Chen's property)}
      For any three $n<n'<n''\in\NZero$ we have
      \[
        \DS(x)_{n,n'}\centerdot \DS(x)_{n',n''}=\DS(x)_{n,n''}.
      \]
  \end{enumerate}
\end{theorem}

\begin{remark}\label{rmk:multidimQsym}
  1. Observe that point (i) in Theorem \ref{theorem:quasiShuffleChen} amounts to a generalisation of the algebra isomorphism defined in \eqref{Qsym} to the multidimensional case, i.e., for an alphabet $A=\{\letter 1, \ldots,\letter d\}$. Indeed, defining the map $\Sigma_d$ on $H_{\mathrm{qsh}}$
\begin{equation}
\label{dQsym}
	\Sigma_d([u_1][u_2]\dotsm[u_n])
	\coloneqq \sum_{1 \le j_1 < \cdots < j_n} Y^{[u_1]}_{j_1} \cdots Y^{[u_n]}_{j_n},
\end{equation}
where $u_1,\ldots,u_n \in \mathfrak A$ and for $u=[\letter a_1 \cdots \letter a_l] \in \mathfrak A$ we have set
$$
	Y^{[u]}_{j} := Y^{[\letter a_1]}_{j} \cdots Y^{[\letter a_l]}_{j},
$$
we obtain a quasi-shuffle algebra  isomorphism into the algebra of \emph{quasi-symmetric functions of level $d$}, as introduced by Novelli and Thibon in \cite{NovThi2004}. For the sake of briefness we only remark that
\begin{align*} 
	\langle w, \DS(x) \rangle
  	= \Sigma_d(w)\Big\rvert_{Y_1^{[\letter 1]} = \Delta x_1^{[\letter 1]}, \dots, 
	Y_1^{[\letter d]} = \Delta x_1^{[\letter d]}, Y_2^{[\letter 1]} = \Delta x_2^{[\letter 1]}, \dots}
\end{align*}

2. Specialising to $\field = \mathbb R$, Theorem \ref{theorem:quasiShuffleChen}
matches the corresponding result for the iterated-integrals signature $S(X)$ of a curve of bounded variation in $\R^B$, where $B$ is a (possibly countable) alphabet. The iterated-integrals signature is also called {\textit{Chen's signature}}, {\textit{rough path signature}}, {\textit{continuous-time signature}} or just {\textit{signature}} in the literature.

Here, the underlying Hopf algebra is $H_{\shuffle}=(T(B),\shuffle,\delta,\varepsilon)$.
Indeed (see for example \cite{HaiKel2014}),
\begin{enumerate}
  \item \textbf{(Shuffle identity)} For fixed $s < t$, $S(X)_{s,t}$ is a character on $H_\shuffle$, that is
    for all $v,w \in H_\shuffle$
    \begin{align*}
      \langle v, S(X)_{s,t} \rangle \cdot \langle w, S(X)_{s,t} \rangle
      =
      \langle v \shuffle w, S(X)_{s,t} \rangle.
    \end{align*}
  \item \textbf{(Chen's property)} For $s < u < t$
    \begin{align*}
      S(X)_{s,u} \centerdot S(X)_{u,t} = S(X)_{s,t}.
    \end{align*}
\end{enumerate}
\end{remark}

Before proving Theorem \ref{theorem:quasiShuffleChen} we need the following abstract result, which is a particular case of the setting presented in \cite[Section 5.1]{NovThi2004}.

\begin{lemma}\label{lemma:sigasgen}
  Let $M_{[u_1]\dotsm[u_k]}(Y)\coloneqq\Sigma_d([u_1]\dotsm[u_k])$ denote the level $d$ monomial quasisymmetric functions defined in \eqref{dQsym}. Then, the
  ``generating series''
  \[
  	\sigma(Y)\coloneqq\sum_{w\in \mathfrak A} M_w(Y)\, w
  \]
  admits the factorisation
\begin{equation}
\label{factorisation}
  	\sigma(Y)=\vec{\prod_{j\ge1}}\left( \varepsilon
	-\sum_{\letter a\in A}Y_j^{[\letter a]}\, [\letter a] \right)^{-1}
	=\vec{\prod_{j\ge1}}\left( \varepsilon+\sum_{[u]\in\mathfrak A}Y_j^{[u]}\, [u] \right).
\end{equation}
\end{lemma}

Let us look at the first few terms in \eqref{factorisation}:  
\begin{align*}
	\lefteqn{\sigma(Y)=\left( \varepsilon+\sum_{[u]\in\mathfrak A}Y_1^{[u]}\,[u] \right)
				\left( \varepsilon+\sum_{[u]\in\mathfrak A}Y_2^{[u]}\,[u] \right)
					\left( \varepsilon+\sum_{[u]\in\mathfrak A}Y_3^{[u]}\,[u] \right)\cdots}\\
		      &= \varepsilon 
		      		+ \sum_{[u]\in\mathfrak A }\big(Y_1^{[u]} + Y_2^{[u]} + Y_3^{[u]} + \cdots\big)\,[u] 
				+ \sum_{[u][v] \in \mathfrak{A}^{\otimes 2}}\big(Y_1^{[u]}Y_2^{[v]} + Y_1^{[u]}Y_3^{[v]} + \cdots\big)\,[u][v] + \cdots     
\end{align*}
Instead of elaborating on this lemma, we refer to
reference \cite{NovThi2004} for details about multivariable generating series.
Note, however, that after evaluating $\sigma(Y)$ in $Y_j^{[\letter a]}=\Delta
x_j^{[\letter a]}$, we obtain $\DS(x)$ and the factorisation
\eqref{factorisation} takes place in the convolution algebra $(T\dlrp{\mathfrak
A},\centerdot)$. We further remark that the expansion of the geometric series
on the righthand side of the first equality in \eqref{factorisation} takes
place in $\mathfrak A$, which explains the summation over $\mathfrak A$ in the
second equality.

\begin{remark}
  Equality \eqref{factorisation} bears resemblance to \cite[Definition 4.1]{KiralyOberhauser2019}
  (c.f.~also \cite[Theorem 32]{LyonsOberhauser2017}. We would like to thank Harald Oberhauser (Oxford) for pointing us to these references).
  At first sight though, only coefficients for words
  in letters of weight one are considered in the aforementioned reference (e.g.~in our notation $[1],[2],\dotsc,[d],[1][1],[1][2],\dots,[1][1][1],$\ldots).
  Preprocessing the underlying time series through a nonlinear function (i.e.~a kernel in the terminology of \cite{KiralyOberhauser2019}), one \emph{can} introduce additional polynomial expressions. But, note that in their setting then nonetheless sums of \emph{increments of polynomials} appear,  whereas in the iterated-sums signature (i.e.~in \eqref{factorisation} evaluated at $Y_i = \Delta x_i$) \emph{polynomials of increments} show up.

  The differences between the two approaches may be summarized
  by saying that increments of polynomials differ from polynomials of increments.
  Saying this, it is an interesting question how these two approaches could be combined fruitfully.
  In particular, we hope to investigate the application of kernelization techniques to the iterated-sums signature.


  Finally, we would like to mention that the
  work of Hoffman--Ihara (see \Cref{sec:HoffSig} and \cite{HofIha2017}, as well as \cite{FoiPat2016})
  permits to define for any positive integer a linear
  automorphism of $T(\mathfrak A)$ which gives rise to a family of ``feature maps'' interpolating between the iterated-sums signature and the
  iterated-integrals signature.
  This relates to a modification of \eqref{factorisation} in the spirit of \cite[Appendix B]{KiralyOberhauser2019}.
  These new feature maps define characters over
  Hopf algebras equipped with new quasi-shuffle type products. The corresponding
  family of linear automorphisms define algebra maps between these quasi-shuffle
  type products and the quasi-shuffle product \eqref{quasishuffle}.   
  We postpone the details of this construction to a follow-up paper,
  and would like to thank the anonymous referee for hinting at this direction.
\end{remark}


\begin{proof}[Proof of Theorem \ref{theorem:quasiShuffleChen}]
  1.
We need to show that for words $w,w' \in T(\mathfrak A)$
$$
	\langle w\star w',\DS(x)_{n,m}\rangle
  	=\langle w,\DS(x)_{n,m}\rangle\langle w',\DS(x)_{n,m}\rangle.
$$
We use the recursive definition of the quasi-shuffle product \eqref{quasishuffle} and induction on $\ell(w)+\ell(w')$, the base case (i.e., $w=e$ or $w'=e$) being trivial. If $u,v\in T(\mathfrak A)$ and $a,b\in\mathfrak A$, define the auxiliary time series
\[
	s_k  \coloneqq \langle u[a],\DS(x)_{n,n+k}\rangle,
	\quad
	s'_k \coloneqq\langle v[b],\DS(x)_{n,n+k}\rangle
\]
for $0\le k\le m-n$, and zero else. Observe that the increments
\[
	(\Delta s)_k=\langle u,\DS(x)_{n,n+k-1}\rangle\Delta x_{n+k}^{[a]},
	\quad
	(\Delta s')_k=\langle v,\DS(x)_{n,n+k-1}\rangle\Delta x_{n+k}^{[b]}.
\]
By the induction hypothesis we then get
\allowdisplaybreaks
\begin{align*}
      s'_{k-1}(\Delta s)_k
      &=\langle u,\DS(x)_{n,n+k-1}
      \rangle\langle v[b],\DS(x)_{n,n+k-1}\rangle\Delta x_{n+k}^{[a]}\\
      &=\langle u\star v[b],\DS(x)_{n,n+k-1}\rangle\Delta x_{n+k}^{[a]},
\end{align*}
and similarly
\begin{align*}
      s_{k-1}(\Delta s')_k
      &=\langle u[a],\DS(x)_{n,n+k-1}
      \rangle\langle v,\DS(x)_{n,n+k-1}\rangle\Delta x_{n+k}^{[b]}\\
      &=\langle u[a]\star v,\DS(x)_{n,n+k-1}\rangle\Delta x_{n+k}^{[b]}.
\end{align*}
Also, by a similar argument we also have
\[
	(\Delta s)_k(\Delta s')_k=\langle u\star v,\DS(x)_{n,n+k-1}\rangle \Delta x_{n+k}^{[a]} \Delta x_{n+k}^{[b]}.
\]
Finally, we summate these relations by using Lemma \ref{lemma:discdiff} to get
\begin{align*}
       \lefteqn{\langle u[a],\DS(x)_{n,m}\rangle\langle v[b],\DS(x)_{n,m}\rangle
      	=\sum_{k=1}^{m-n}\Delta(ss')_k}\\
      &=\langle(u\star v[b])[a]+(u[a]\star v)[b]+(u\star v)[ab],\DS(x)_{n,m}\rangle\\
      &=\langle u[a]\star v[b],\DS(x)_{n,m}\rangle.
\end{align*}

2.~The proof of Chen's property can be pursued using a pedestrian approach. However, it also follows from Lemma \ref{lemma:sigasgen}. Indeed, we may split the product in the factorisation \eqref{factorisation} as
\[
	\sigma(Y)=\vec{\prod_{1\le j\le n'}}\left( \varepsilon-\sum_{\letter a\in A}Y_j^{[\letter a]}[\letter a] \right)^{-1}
	\centerdot
	\vec{\prod_{j>n'}}\left( \varepsilon-\sum_{\letter a\in A}Y_j^{[\letter a]}[\letter a] \right)^{-1}.
\]
The desired identity follows upon evaluation at $Y_j^{[\letter a]}=\Delta x_j^{[\letter a]}$ as in the previous remark.
\end{proof}

We note that the iterated-sums signature, $\DS(x)_{n,m}$, introduced in this work is similar to the discrete Chen(--Fliess) series defined and studied in \cite{Gray2017} in the context of nonlinear control theory.

\smallskip

This section is closed with an intriguing observation. Up to this point it may seem that iterated-sums signatures, $\DS(x)_{n,m}$, and Chen's signatures, $S(X)_{s,t}$ (see Remark \ref{rmk:multidimQsym}), behave in the same way, but as the next example shows this is not at all the case. Recall that $\mathrm{End}_\field(H_{\mathrm{qsh}})$, the space of linear maps on $H_{\mathrm{qsh}}$, together with the convolution product $\psi \ast \gamma \coloneqq m_{\star}(\psi \otimes \gamma)\delta,$ is a non-commutative algebra with unit $\iota \coloneqq \eta \circ \varepsilon$, where $\eta: \field \to H_{\mathrm{qsh}}$ is the unit map, $\eta(\lambda):=\lambda e$. Define
\begin{equation}
\label{adjeulerian}
	\mathfrak e\coloneqq\log^*(\id)=J-\frac12J*J+\frac13J*J*J+\dotsb,
\end{equation}
where $J\coloneqq \id - \iota \in \mathrm{End}_\field(H_{\mathrm{qsh}})$ is the projection onto the augmentation ideal $T_+(\mathfrak A)$. It is the adjoint of the classical Eulerian Lie idempotent \cite{Reutenauer1993}, that is, the concatenation logarithm of the identity map, $\log^{\centerdot}(\id)$. Observe that the sum \eqref{adjeulerian} terminates when evaluated in homogeneous elements since $J(e)=0$, thus it is well defined for arbitrary elements of $T(\mathfrak A)$. Then, for any character $c \in T\dlrp{\mathfrak A}$ and word $u\in T(\mathfrak A)$ we have that
\[
	\langle u,\log^\centerdot c\rangle=\langle\mathfrak e(u),c\rangle,
\]
where $\log^\centerdot c \in g$. Indeed, by definition
\begin{align*}
\langle\mathfrak e(u),c \rangle
  	 &=\sum_{k=1}^\infty\frac{(-1)^{k-1}}{k}\langle J^{*k}(u),c \rangle
	  =\sum_{k=1}^\infty\frac{(-1)^{k-1}}{k}\sideset{}{'}\sum_{(u)}
  	 	\langle u_{(1)} \star \cdots \star u_{(k)}  ,c \rangle\\
  	&=\sum_{k=1}^\infty\frac{(-1)^{k-1}}{k}\sideset{}{'}\sum_{(u)}
  	 	\langle u_{(1)} \otimes \cdots \otimes u_{(k)}  ,c ^{\otimes k}\rangle
	  =\sum_{k=1}^\infty\frac{(-1)^{k-1}}{k}\langle u,(c -\varepsilon)^{\centerdot k}\rangle\\
  	&=\langle u,\log^{\centerdot}c \rangle.
\end{align*}
In the third equality we used that $c$ is a character. In the second equality the reduced coproduct is applied
$$
	\langle J^{*k}(u),c \rangle
	= \langle m_\star (J \otimes J^{*k-1})\delta'(u),c \rangle
	= \langle \sideset{}{'}\sum_{(u)} u_{(1)} \star J^{*k-1}(u_{(2)}),c \rangle.
$$
Now, if $x$ is an arbitrary time series, for its iterated-sums signature this means that
\[
	\langle[\letter 1^2],\log^\centerdot \DS(x)\rangle
	=\langle[\letter 1^2],\DS(x)\rangle
	=\sum_j\left(\Delta x_j^{[1]}\right)^2\ge0.
\]
Therefore, the image of the logarithm of iterated-sums signatures only reaches a certain subset of the Lie algebra of infinitesimal characters on $H_{\mathrm{qsh}}$. This is in contrast to Chen's iterated-integrals signature, for which Chow's Theorem \cite[Theorem 7.28]{FrizVic2010} holds, showing that any character over the shuffle Hopf algebra may be realised as the Chen signature of a piecewise linear path. The implications of this observation will be studied in a forthcoming paper.

Still, the following positive statement on the \emph{linear span} of iterated-sums signatures holds.

\begin{lemma}
  \label{lemma:spanning}
  For every $n\ge 1$, $\operatorname{span}_\field \{ \proj_{\le n} \DS(x) : x \in (\field^d)^\NOne_c\} = \proj_{\le n} H^*_{\mathrm{qsh}}$.
\end{lemma}

\begin{remark}
  The corresponding result for iterated-integrals signatures was shown in \cite[Lemma 3.4]{DiehlReizenstein2018},
  which is sometimes useful for proving statements
  about the underlying algebra that are easily verified
  when tested against signatures.
\end{remark}

\begin{proof}
  Fix $n \ge 1$
  and let $P_1, \dots, P_L$, ordered in some way, be the quasisymmetric
  monomial functions with degree smaller or equal to $n$.

  By \cite[Section 5.1]{NovThi2004}
  they are independent as elements of the space of formal power series.
  
  This implies that, for some $m\ge 1$ large enough, evaluating at $Y_m = (y_1, \ldots, y_m, 0, 0, \ldots )$,
  the expressions $P_1(Y_m), \dots, P_L(Y_m)$ are independent as elements of $\field\left[ y_1, \ldots, y_m \right]$.
  Denote
  \begin{align*}
    Y^{(\ell)}_m = \left(y_1^{(\ell)}, \dots, y^{(\ell)}_m, 0, 0, \ldots\right),\quad \ell = 1,\ldots, L,
  \end{align*}
  copies in new variables of $Y_m$,
  that is $P_i\left( Y^{(\ell)}_m \right) \in R := \field\left[ y^{(1)}_1, \dots, y^{(1)}_m, \dots,  y^{(L)}_1, \ldots, y^{(L)}_m \right]$.
  Then, the independence of the $P_i$ implies independence, in $R$, of the \emph{rows} of
  \begin{align*}
    \begin{pmatrix}
      P_1\left( Y^{(1)}_m \right) & P_1\left( Y^{(2)}_m\right) & \dots & P_1\left( Y^{(L)}_m \right) \\
      P_2\left( Y^{(1)}_m \right) & P_2\left( Y^{(2)}_m\right) & \dots & P_2\left( Y^{(L)}_m \right) \\
      \vdots& \vdots& \ddots& \vdots \\
      P_L\left( Y^{(1)}_m \right) & P_L\left( Y^{(2)}_m\right) & \dots & P_L\left( Y^{(L)}_m \right)
    \end{pmatrix}.
  \end{align*}
  A fortiori, also the \emph{columns} must be independent in $R$. Hence, the columns must be independent for \emph{some} realisation of the $Y^{(\ell)}$.
  This finally implies that we can find $x^{(\ell)} \in (\field^{d})^{m+1}, \ell =1, \ldots, L$ such that the columns of
  \begin{align*}
    \begin{pmatrix}
      P_1\left( \Delta x^{(1)} \right) & P_1\left( \Delta x^{(2)}\right) & \dots & P_1\left( \Delta x^{(L)} \right) \\
      P_2\left( \Delta x^{(1)} \right) & P_2\left( \Delta x^{(2)}\right) & \dots & P_2\left( \Delta x^{(L)} \right) \\
      \vdots& \vdots& \ddots& \vdots \\
      P_L\left( \Delta x^{(1)} \right) & P_L\left( \Delta x^{(2)}\right) & \dots & P_L\left( \Delta x^{(L)} \right)
    \end{pmatrix},
  \end{align*}
  are independent. Here, as before, we extend the $x^{(\ell)}$ constantly to an element of $(\field)^{\NOne}_c$.
\end{proof}


\section{Invariants}
\label{sec:inv}

In the previous section, we defined the iterated-sums signature, following the introduction. We now return to our original motivation, and first put the concept of ``time warping'' in a precise mathematical framework.
For each index $n\ge0$ we define an operator acting on sequences by repeating once the value at time $n$. More precisely, given a time series, $x$, we define $\tau_n(x)$ as the time series given by
\[
	\tau_n(x)_j\coloneqq\begin{cases}
					x_j&j\le n\\
				   x_{j-1}&j>n
				   \end{cases}.
\]
Observe that with this definition we have $\tau_n(x)_n=\tau_n(x)_{n+1}=x_n$, and the rest of the values are unchanged save for a time shift after time $n$.

\begin{definition}
  We call a functional $F: (\field^d)^{\NOne}_c \to \field$ \textbf{invariant to time warping} if $F\circ\tau_n=F$ for all $n\ge1$.
\end{definition}
From applications to data analysis, such as, e.g., moment corrections, we are
mostly interested in polynomial invariants, i.e., invariant functionals that can
be expressed by considering only polynomial expressions in a time series.

\begin{definition}
  We call $F: (\field^d)^{\NOne}_c \to \field$ \textbf{polynomial}, if for all $N \ge 1$, $F(x_0,\dotsc,x_N, 0,0,\dotsc)$
  is a polynomial in the $x_i$, and the polynomial degree is uniformly bounded in $N$.
\end{definition}

From the factorisation \eqref{factorisation} in Lemma \ref{lemma:sigasgen} it follows that for any word $w\in T(\mathfrak A)$ the coefficient $\langle w,\DS(x)\rangle$ is a polynomial invariant in this sense.
It turns out these are all the polynomial invariants, if we additionally demand invariance with respect to space translation of the entire series.

\begin{lemma}
  \label{lem:allInvariants}
  Let $F$ be polynomial, invariant to both time warping and
  space translations. Then: $F$ is realised as a quasisymmetric function.
\end{lemma}

\begin{proof}
  We do the one-dimensional case, $d=1$, to avoid notational clutter.
  By translation invariance, for any $N \ge 1$,
  \begin{align*}
    F(x_0, x_1, \ldots, x_N, 0, 0, \ldots) = F(0, x_1-x_0, x_2 - x_0,\ldots, x_N - x_0, 0, 0,\ldots).
  \end{align*}
  Now, by assumption, this is a polynomial in $x_1-x_0, x_2 - x_0,\ldots, x_N- x_0,$ hence it is a (different) polynomial in  $x_1-x_0, x_2-x_1, x_3-x_2,\ldots,x_N-x_{N-1}$. Therefore, $F$ can be realised as a formal power series of bounded degree:
  there is $\hat F \in \field\langle Y_1, Y_2, \dots \rangle$ of bounded degree such that
  for $x \in (\field^d)^{\NOne}_c$ we have that $F( x ) = \hat F( \Delta x )$.

  It remains to show that $\hat F$ is quasisymmetric.
  Let $n \ge 1$, $i_1 < \cdots < i_n$ and $\alpha_1, \dots, \alpha_n \ge 1$.
  We show that the coefficient of the monomial $Y_{i_1}^{\alpha_1} \cdots Y_{i_n}^{\alpha_n}$ in $\hat F$ is equal to the one of $Y_{1}^{\alpha_1} \cdots Y_{n}^{\alpha_n}$.

  Indeed: by using repeatedly the invariance to time warping, we get that
  for all $x \in \R^n$,
  \begin{align*}
    \hat F(\Delta x_1,\Delta x_2,\dots, \Delta x_n,0,0,\ldots)
    =
    \hat F(0,\ldots,0,\underbrace{\Delta x_1}_{i_1},0,\dots,0,\underbrace{\Delta x_2}_{i_2},0, \dots, 0, \underbrace{\Delta x_n}_{i_n}, 0, \ldots).
  \end{align*}
  Hence, both sides coincide \emph{as polynomials}.
  So that the coefficient of
  $Y_{i_1}^{\alpha_1} \cdots Y_{i_n}^{\alpha_n}$
  and
  $Y_1^{\alpha_1} \dots Y_{n}^{\alpha_n}$ must coincide.
  This finishes the proof.
\end{proof}


\section{Hoffman's isomorphism and signatures}
\label{sec:HoffSig}

In this section we relate the iterated-sums signature of a time series with the usual iterated-integrals signature of the piecewise linear interpolation of an associated \emph{infinite dimensional} time series.

Starting again with the extended alphabet $\mathfrak A$, we build the tensor algebra $T(\mathfrak A)$ and define the shuffle product $\shuffle : T(\mathfrak A)\otimes T(\mathfrak A)\to T(\mathfrak A)$ inductively by
\[
	u[a]\shuffle v[b] \coloneqq (u \shuffle v[b])[a]+(u[a]\shuffle v)[b].
\]
Recall Hoffman's isomorphism \cite{Hoffman2000} defined in Theorem \ref{thm:HoffExp}, which shows that $H_{\shuffle}=(T(\mathfrak A),\shuffle,\delta)$ and $H_{\mathrm{qsh}}=(T(\mathfrak A),\star,\delta)$ are isomorphic as Hopf algebras. Next we compute explicitly the image by the iterated-integrals signature $S$ of a linear path.

The following lemma is an immediate extension of \cite[Example 7.21]{FrizVic2010} to a countable index set.

\begin{lemma}\label{lemma:linesig}
  Consider a countable set $B$ and let $z_t=z_0+at$ for some $z_0,a\in\mathbb R^B$ and all $t\in[0,1]$. Then for $w=w_1 \cdots w_n \in T(B)$
  \[
    \langle w,S(z)_{s,t}\rangle=\frac{(t-s)^{\ell(w)}}{\ell(w)!}\prod_{j=1}^{\ell(w)}a^{w_j}.
  \]
\end{lemma}

At the level of the tensor algebra this simply means that $S(z)_{s,t}=\exp^\centerdot((t-s)a)$. An analogue of this result holds for discrete signatures, which follows from Lemma \ref{lemma:sigasgen}, i.e., Chen's property.

\begin{lemma}\label{lemma:singlestepdsig}
  Let $x=(0,v,v,\dotsc)$ be a time series having a single non-zero increment $v=(v^{[\letter1]},\dotsc,v^{[\letter d]})\in\mathbb R^d$. Then
\[
	\langle w,\DS(x)\rangle
	=\begin{cases}
    		(v^{[\letter 1]})^{k_1}\dotsm (v^{[\letter d]})^{k_d},
    			&w=[\letter 1^{k_1}\dotsm\letter d^{k_d}]\\
		0,
			&\text{else}
	  \end{cases}.
\]
  \end{lemma}

  Now we look for a relation between the iterated-integrals signature and the iterated-sums signature. For this, let $x=(0,x_1,x_2,\dotsc)$ be a time series and consider the (infinite dimensional!) path $X=(X^a:a\in\mathfrak A)$ where, for $a=[\letter 1^{k_1}\dotsm\letter d^{k_d}]\in\mathfrak A$, the component path $X^{a}$ is the linear interpolation of the time series
  \begin{align}
    \label{eq:additionalTerms}
	  n\mapsto\sum_{j=1}^n\Delta x_j^a
	  =\sum_{j=1}^n(\Delta x_j^{[\letter 1]})^{k_1}\dotsm(\Delta x_j^{[\letter d]})^{k_d}.
  \end{align}

\begin{theorem}
  \label{theorem:hoffmanDoesTheJob}
  We have $\langle\Phi_{\mathrm{H}}(w),\DS(x)\rangle = \langle w,S(X)\rangle$.
\end{theorem}

\begin{remark}
  We note that the iterated-integrals signature of the $d$-dimensional path consisting in the piecewise linear interpolation of $x$ is \emph{not} enough to obtain $\DS(x)$.
  Instead, the theorem shows that the iterated-integrals signature of the piecewise linear interpolation
  of the infinite dimensional time series \eqref{eq:additionalTerms} is sufficient.
\end{remark}

\begin{proof}
  Without loss of generality let the interpolation of \eqref{eq:additionalTerms} happen at the time points $0,1,2,\dots, N$.
  Then, by Chen's property,
  \[
        S(X) = \exp^\centerdot(X_1)\centerdot\exp^\centerdot(X_2-X_1)
    \centerdot\dotsm\centerdot\exp^\centerdot(X_N-X_{N-1}).
  \]

  We first investigate what happens for a single time step.
  Let a word $w=[a_1]\dotsm [a_p]\in T(\mathfrak A)$ be given,
  and write $a_i = [\letter 1^{k_1^i}\dotsm\letter d^{k_d^i}] \in\mathfrak A$, $i=1,\dotsc,p$.
  According to Lemma \ref{lemma:linesig}, 
  \begin{align*}
        \langle w,\exp^\centerdot(X_j-X_{j-1})\rangle
        &=\frac{1}{p!}\Delta x_j^{[a_1]}\dotsm \Delta x_j^{[a_p]}
        =\frac{1}{p!}(\Delta x_j^{[\letter 1]})^{K_{\letter 1}}
        \dotsm(\Delta x_j^{[\letter d]})^{K_{\letter d}}
        =\frac{1}{p!}\Delta x_j^{[\letter 1^{K_{\letter 1}}\dotsm\letter d^{K_{\letter d}}]}.
  \end{align*}
  where $K_{\letter m}\coloneqq k_{\letter m}^1+\dotsb+k_{\letter m}^p$.
  In other words, $K_{\letter m}$ is the number of times the letter $\letter m\in\{\letter 1,\dotsc,\letter d\}$ is repeated in $w$.

  Now the only term in $\Phi_{\mathrm{H}}(w)$ containing a single letter is
  $\frac{1}{p!}[\letter 1^{K_{\letter 1}}\dotsm\letter d^{K_{\letter d}}]$, i.e.,
  the full ``contraction''. Then, by Lemma \ref{lemma:singlestepdsig},
  \begin{align*}
    \langle \Phi_{\mathrm{H}}(w), \DS(x)_{j-1,j} \rangle
    &=
    \langle \frac{1}{p!}[\letter 1^{K_{\letter 1}}\dotsm\letter d^{K_{\letter d}}], \DS(x)_{j-1,j} \rangle \\
    &=
    \frac{1}{p!}\Delta x_j^{[\letter 1^{K_{\letter 1}}\dotsm\letter d^{K_{\letter d}}]} \\
    &=
    \langle\Phi_{\mathrm{H}}(w),\DS(x)_{j-1,j}\rangle.
  \end{align*}
  Therefore, we have shown the claim for a single time step.

  Now, since $\Phi_{\mathrm{H}}$ is a Hopf algebra map, the statement of the theorem is equivalent to showing that  $\Phi_{\mathrm{H}}^*(\DS(x))=S(X)$, where $\Phi_{\mathrm{H}}^*$ is the adjoint of Hoffman's isomorphism.
  Since $\Phi_{\mathrm{H}}^*$ is an algebra morphism, we calculate
  \begin{align*}
    \Phi_{\mathrm{H}}^*(\DS(x)_{0,N})&=\Phi_{\mathrm{H}}^*(\DS(x)_{0,1})\centerdot\dotsm\centerdot\Phi_{\mathrm{H}}^*(\DS(x)_{N-1,N})=S(X)_{0,1}\centerdot\dotsm\centerdot S(X)_{N-1,N}\\
          &=S(X)_{0,N}.
  \end{align*}
  So the result is valid for the full signature.
\end{proof}

Finally, we show a consistency result.
\begin{proposition}
  Let $X\colon[0,1]\to\R^d$ be a continuous path of finite variation, meaning that
  \[ 
  	\sup_{\pi}\sum_{[s,t]\in\pi}\lVert X_t-X_s\rVert<\infty 
  \]
  where the supremum is taken over all partitions $\pi$ of $[0,1]$.

  Given such a partition $\pi=\{t_0=0<t_1<\dotsb<t_{N-1}<t_N=1\}$, define $x(\pi)$ by $x(\pi)_j=X_{t_j}$.
  Then
  \[ 
  	\lim_{\lvert\pi\rvert\to 0}\langle w,\DS(x(\pi))_{0,N}\rangle=
	\begin{cases}
	\langle w,S(X)_{0,1}\rangle &w\in T(A)\\ 
	0&w\not\in T(A)
	\end{cases}. 
	\]
\end{proposition}
\begin{proof}
  We use induction on the length $\ell(w)$. If $\ell(w)=1$ and $w=\letter i\in A$, then 
  \[ \langle\letter i,\DS(x(\pi))_{0,N}\rangle=\sum_{[s,t]\in\pi}( x(\pi)^{\letter i}_t-x(\pi)^{\letter i}_s )=X^i_1-X^i_0=\int_0^1\mathrm dX^i_s \]
  which is independent of $\pi$.
  If, on the other hand, $w=a\in\mathfrak A\setminus A$ then $a=[\letter 1^{k_1}\dotsm\letter d^{k_d}]$ with $k_1+\dotsb+k_d\ge 2$.
  Therefore
  \begin{align*}
    \lvert\langle a,\DS(x(\pi))_{0,N}\rangle\rvert&= \left\lvert \sum_{j=1}^N\prod_{i=1}^d(\Delta x(\pi)_j^i)^{k_i} \right\rvert\\
    &\le \sum_{j=1}^N\prod_{i=1}^d\lvert\Delta x(\pi)_j^i\rvert^{k_i}\\
    &= \sum_{j=1}^N\prod_{i=1}^d\left\lvert X^i_{t_j}-X^{i}_{t_{j-1}} \right\rvert^{k_i}\\
    &= \sum_{j=1}^N\lVert X_{t_j}-X_{t_{j-1}} \rVert^{k_1+\dotsb+k_d}\\
    &\le \lVert X\rVert_1\sup_{j=1,\dotsc,N}\lVert X_{t_j}-X_{t_{j-1}} \rVert^{k_1+\dotsb+k_d-1}
  \end{align*}
  which vanishes in the limit since $X$ is uniformly continuous on $[0,1]$.

  Now suppose $w=w'a$ for some $w\in T(\mathfrak A)$ and $a\in\mathfrak A$.
  We have 3 cases
  \begin{enumerate}
    \item $w'\in T(\mathfrak A\setminus A)$: in this case, no matter what $a$ is, we have
      \begin{align*}
        \lvert\langle w,\DS(x(\pi))_{0,N}\rangle\rvert&\le\sum_{j=1}^N\lvert\langle w',\DS(x(\pi))_{0,j}\rangle\rvert\lvert\Delta x(\pi)_j^a\rvert\\
        &\le\sup_{j=1,\dotsc,N}\lVert X_{t_j}-X_{t_{j-1}}\rVert^{|a|}\sum_{j=1}^N\lvert\langle w',\DS(x(\pi))_{0,j-1}\rangle\rvert\to 0
      \end{align*}
      as $\lvert\pi\rvert\to 0$, by the induction hypothesis.
    \item $w'\in T(A)$ and $a\in\mathfrak A\setminus A$: the same argument as before gives that the corresponding entry in $\DS(x(\pi))$ vanishes in the limit.
    \item $w'\in T(A)$ and $a\in A$: again by definition we have
      \[ \langle w,\DS(x(\pi))_{0,N}\rangle=\sum_{j=1}^N\langle w',\DS(x(\pi))_{0,j-1}\rangle(X^a_{t_j}-X^a_{t_{j-1}}) \]
      which converges to the Young (or Riemann--Stieltjes) integral
      \[ \int_0^1\langle w',S(X)_{0,s}\rangle\,\mathrm dX^a_s=\langle w,S(X)_{0,1}\rangle. \] 
  \end{enumerate}
  Therefore, we have that
  \[ \lim_{|\pi|\to 0}\langle w,\DS(x(\pi))_{0,N}\rangle=\langle w,S(X)_{0,1}\rangle \]
  if $w\in T(A)$, and vanishes otherwise.
\end{proof}


\subsection{The area operation}
\label{ssec:area}

It is well known that for the iterated-integrals signature certain linear combinations of the entries have a precise geometric interpretation.
Indeed, for any $\letter i,\letter j\in A$
\[
	\langle\letter{ij}-\letter{ji}, S(X)_{s,t}\rangle
	=\iint\limits_{s<u_1<u_2<t}(\mathrm dX^{\letter i}_{u_1}\mathrm dX^{\letter j}_{u_2}
  -\mathrm dX^{\letter j}_{u_1}\mathrm dX^{\letter i}_{u_2})\eqqcolon\Area(X^{\letter i},X^{\letter j})_{s,t},
\]
represents (two times) the signed area (or L\'evy area) between the curves
$u\mapsto X_u^{\letter i}$ and $u\mapsto X_u^{\letter j}$ for $u\in[s,t]$, and
the cord between the points $(X_s^{\letter i},X_s^{\letter j})$ and
$(X_t^{\letter i},X_t^{\letter j})$.

We abstract this operation to the shuffle algebra by using the notion of half-shuffles introduced in \Cref{ssec:halfshuf}.
In fact, one verifies that at this level the area operation may be represented in terms of half-shuffle operations as
$$
  \letter{ij}-\letter{ji}=\letter i\succ\letter j-\letter j\succ\letter i\eqqcolon\area(\letter i,\letter j),
$$
so that in particular $\Area(X^{\letter i},X^{\letter j})_{s,t}=\langle\area(\letter i,\letter j),S(X)_{s,t}\rangle$.

We extend this by defining area operations on $H_{\shuffle}=(T(\mathfrak A),\shuffle,\delta)$ and $H_{\mathrm{qsh}}=(T(\mathfrak A),\star,\delta)$.

\begin{definition}
\label{def:sharea}
The \emph{area map} $\area\colon H_{\shuffle}\otimes H_{\shuffle}\to H_{\shuffle}$ is defined by
\[
	\area(u,v)\coloneqq u\succ v-v\succ u.
\]
\end{definition}

Next, the discrete analogue is given in terms of the first half-shuffle product in \eqref{ex:halfshuf}.

\begin{definition}[Discrete area]
\label{def:qsharea}
The \emph{discrete area map} $\darea\colon H_{\mathrm{qsh}}\otimes H_{\mathrm{qsh}}\to H_{\mathrm{qsh}}$ is defined by
\[
	\darea(u,v)\coloneqq u\qsucc v-v\qsucc u.
\]
\end{definition}

We compare the two areas by considering the words $u=[\letter 3]$ and $v=[\letter 4][\letter{12}]$. Then
\begin{align*}
    \area([\letter 3],[\letter 4][\letter{12}])
    &=[\letter{3}][\letter{4}][\letter{12}]
    	+[\letter{4}][\letter{3}][\letter{12}]
    	-[\letter 4][\letter{12}][\letter 3],\\
    \darea([\letter 3],[\letter 4][\letter{12}])
    &=[\letter{3}][\letter{4}][\letter{12}]
      +[\letter{4}][\letter{3}][\letter{12}]
	+[\letter{34}][\letter{12}]
	-[\letter 4][\letter{12}][\letter 3]
\end{align*}
as follows from Example \ref{ex:qshdend}.

Both $\area$ and $\darea$ can be iterated. We now make this precise:
define $\mathsf D_1=D_1 \coloneqq \field\,\mathfrak A$, the vector space spanned by the set $\mathfrak A$. Then, inductively define vector spaces
\begin{align*}
  D_{n+1} &\coloneqq \operatorname{span}_\field \{ \area(D_{n+1-m},,D_{m}) : m \le n \} \\
  \mathsf D_{n+1} &\coloneqq \operatorname{span}_\field \{ \darea(\mathsf D_{n+1-m},,\mathsf D_{m}) : m \le n \}.
\end{align*}
We finally set
\[
	D\coloneqq\bigoplus_{n\ge1}D_n,
	\quad\quad
	\mathsf D\coloneqq\bigoplus_{n\ge1}\mathsf D_n.
\]

Neither the $\area$ nor the discrete $\darea$ operations are associative. One can show, however, that $\area$ satisfies a fourth-order relation, known as \emph{tortkara}, introduced by Dzhumadil'daev in the 2007 paper \cite{Dzhumadildaev2007}.
In \cite{Dzhumadil2019speciality} the image of iterated applications of the area map is characterised.
(Compare also \cite[Theorem 28]{Reizenstein2019}).
\begin{theorem}[{\cite[Theorem 2.1]{Dzhumadil2019speciality}}]
\label{thm:areaspan}
The space $D$ is spanned by the set
\[
	\mathfrak A\cup\{\,u([a][b]-[b][a]):a,b\in\mathfrak A, u\in T(\mathfrak A)\,\}.
\]
\end{theorem}

From Lemma \ref{lem:Hoffsplit} and Lemma \ref{lmm:HofDzu} we deduce the following morphism property of Hoffman's isomorphism with respect to $\area$ and $\darea$.

\begin{theorem}
\label{thm:HofArea}
$\Phi_{\mathrm{H}}\colon D\to\mathsf D$ is a tortkara morphism,
i.e., for $\varphi, \psi \in D$
\[
	\Phi_{\mathrm{H}}(\area(\varphi,\psi)) =\darea(\Phi_{\mathrm{H}}(\varphi),\Phi_{\mathrm{H}}(\psi)).
\]
\end{theorem}
\begin{remark}
  1. Note that $\Phi_{\mathrm H}$ is \emph{not} a (quasi-)half-shuffle morphism.
  Only the anti-symmetrisation to $\area$ respectively $\darea$ is nicely compatible with it.

  2. The set $D$ (the set of ``areas-of-areas'') is known to generate $T(\mathfrak A)$ as a shuffle-algebra, see \cite{DLPR2020areas}.
  Applied to iterated-integral signatures
  this means that all their information is already contained in areas-of-areas.
  The area operation $(X,Y) \mapsto \Area(X,Y)$ has an immediate geometric interpretation,
  whereas the operation of integration $(X,Y) \mapsto \int X\,\mathrm dY$.%
  \footnote{The authors would be hard-pressed to explain the latter to a non-mathematician,
  whereas the former can be explained by a simple drawing.}
  Moreover, the area operation is related to antisymmetrised lead-lag correlation in time series analysis, see \cite[Section 3.2]{DiehlReizenstein2018}.
  We refer to \cite[Section 6]{DLPR2020areas} for more applications.
\end{remark}

\begin{proof}
  By Dzumadil'daev's theorem (Theorem \ref{thm:areaspan}) it suffices to prove the claim for the case when $\varphi=u([a][b]-[b][a])$ and $\psi=v([c][d]-[d][c])$. We first observe that in this case the $\area$ operation can be written more explicitly:
\[
	\area(\varphi,\psi)
	=\varphi\succ\psi-\psi\succ\varphi
	=(\varphi\shuffle v[c])[d] - (\varphi\shuffle v[d])[c] - (\psi\shuffle u[a])[b] + (\psi\shuffle u[b])[a].
\]
Each of these terms can be further expanded into three terms. For example, the first one equals
\begin{equation*}
	(\varphi\shuffle v[c])[d]
	=(\varphi\shuffle v)[c][d]+(u[a]\shuffle v[c])[b][d]-(u[b]\shuffle v[c])[a][d].
\end{equation*}
In total there are 12 terms, the remaining 9 terms are
\begin{align*}
    -(\varphi\shuffle v[d])[c]&=
    -(\varphi\shuffle v)[d][c]-(u[a]\shuffle v[d])[b][c]+(u[b]\shuffle v[d])[a][c]\\
    -(\psi\shuffle u[a])[b]&=
    -(\psi\shuffle u)[a][b]-(v[c]\shuffle u[a])[d][b]+(v[d]\shuffle u[a])[c][b]\\
    (\psi\shuffle u[b])[a]&=
    (\psi\shuffle u)[b][a]+(v[c]\shuffle u[b])[d][a]-(v[d]\shuffle u[b])[c][a].
\end{align*}
For each of these terms we can find exactly one other term such that their sum is of the form $w([x][y]-[y][x])$, for $[x],[y]\in\{[a],[b],[c],[d]\}$, and thus by Lemma \ref{lmm:HofDzu} the image of this sum has the form $\Phi_{\mathrm{H}}(w[x])[y]-\Phi_{\mathrm{H}}(w[y])[x]$. To summarise, the image $\Phi_{\mathrm{H}}(\area(\varphi,\psi))$ is a linear combination of 6 terms, each of them having the form $\Phi_{\mathrm{H}}(w[x])[y]-\Phi_{\mathrm{H}}(w[y])[x]$. Now, if we pick any $[x]\in\{[a],[b],[c],[d]\}$ there are exactly three terms containing $[x]$ as the last letter. For example, for $[a]$ these terms are
\[
	\Phi_{\mathrm{H}}\Bigl((v[c]\shuffle u[b])[d]\Bigr)[a]
	+\Phi_{\mathrm{H}}\Bigl((\psi\shuffle u)[b]\Bigr)[a]
	-\Phi_{\mathrm{H}}\Bigl( (v[d]\shuffle u[b])[c] \Bigr)[a]
	=\Phi_{\mathrm{H}}(\psi\shuffle u[b])[a] ,
\]
where the last identity is easy to check using that $\psi=v([c][d]-[d][c])$. Applying a similar argument to all letters we see that
\begin{align*}
	\Phi_{\mathrm{H}}(\area(\varphi,\psi))
	&=\Phi_{\mathrm{H}}(\psi\shuffle u[b])[a]
		 -\Phi_{\mathrm{H}}(\psi\shuffle u[a])[b]
		 -\Phi_{\mathrm{H}}(\varphi\shuffle v[d])[c]
		+\Phi_{\mathrm{H}}(\varphi\shuffle v[c])[d]\\
  	&=\Bigl(\Phi_{\mathrm{H}}(\psi)\star\Phi_{\mathrm{H}}(u[b])\Bigr)[a]
			-\Bigl(\Phi_{\mathrm{H}}(\psi)\star\Phi_{\mathrm{H}}(u[a])\Bigr)[b]\\
	&\quad\ 	-\Bigl(\Phi_{\mathrm{H}}(\varphi)\star\Phi_{\mathrm{H}}(v[d])\Bigr)[c]
			+\Bigl(\Phi_{\mathrm{H}}(\varphi)\star\Phi_{\mathrm{H}}(v[c])\Bigr)[d]\\
  	&=\Phi_{\mathrm{H}}(\varphi)\qsucc\Bigl(\Phi_{\mathrm{H}}(v[c])[d]
	-\Phi_{\mathrm{H}}(v[d])[c]\Bigr)
	-\Phi_{\mathrm{H}}(\psi)\qsucc\Bigl(\Phi_{\mathrm{H}}(u[a])[b]
	-\Phi_{\mathrm{H}}(u[b])[a]\Bigr)\\
  	&=\Phi_{\mathrm{H}}(\varphi)\qsucc\Phi_{\mathrm{H}}(\psi)
	-\Phi_{\mathrm{H}}(\psi)\qsucc\Phi_{\mathrm{H}}(\varphi)\\
  	&=\darea(\Phi_{\mathrm{H}}(\varphi),\Phi_{\mathrm{H}}(\psi)).
\end{align*}
\end{proof}


\section{Conclusion}
\label{sec:cinclusions}

In this work we have
\begin{itemize}
  \item introduced a new set of features for multidimensional time series
    consisting in iterated sums (Section \ref{sec:dsig}); 
  \item shown that these features are invariant to time warping and that these in fact are all the (polynomial) invariants in this sense (Section \ref{sec:inv});
  \item described a Hopf algebraic framework to compute these features (Section \ref{sec:QSh});
  \item shown how this setting mirrors the one of iterated-integrals in some aspects and differs in others (Section \ref{sec:QSh}).
\end{itemize}

There are several possible generalisations of our work.
\begin{itemize}
  \item 
        Let $f, g: \field \to \field$ be such that $f(0) = g(0) = 0$.
        Then iterated-sums of the form
        \begin{align*}
          \sum_{i_1 < i_2} f\left( \Delta x_{i_1} \right) g\left( \Delta x_{i_2} \right),
        \end{align*}
        are also invariant to time warping (and analogously for higher order iterated-sums).
        These are, in general, \emph{not} polynomial in the time series anymore,
        but might still be relevant for certain applications.
        For smooth $f,g$ this should be related to the expansion of nonlinear functionals on stochastic word series \cite{Curry2019}, but the non-smooth case (for example $f(x) = x$, $g(x) = |x|$) is particularly interesting.

  \item Multi-parameter data.
        An object of interest are for example ``images''
        $I: [0,N]\times[0,N] \to \R$ and the time warping invariance becomes an invariance to stretching of the image.
\end{itemize}

We are also interested in exploring the possible applications of these invariants in data science.
\begin{itemize}
  \item Retrieval of similar time series, invariant to time warping:
    see \cite{Yi1998Efficient} (and references therein), where it is stated that
    \emph{``the time warping distance \ldots does not lead to any natural features''}.
    The invariants presented in our work \emph{should} provide those missing features,
    but a mathematical rigorous proof of this statement is left for future work.

  \item Statistical inference in problems involving unknown time warping, as in Example \ref{ex:inference}.

  \item Time series clustering:
    the features of this work can be used to cluster time series according to their ``shape'', i.e.,
    independent of time warping.
    Sometimes a ``prototype'' for each cluster is looked after,
    see for example \cite{Petitjean2011}.
      In this case - as in the previous point - reconstruction of a time series from an (averaged) iterated-sums signature
      would be necessary. A detailed study of this ostensibly hard problem is left for future research.
\end{itemize}

We close with some open questions. At the end of Section \ref{sec:dsig} we showed that an equivalent of Chow's theorem does \emph{not} hold for the iterated-sums signature $\DS(x)$.
\begin{itemize}
  \item 
Can we understand $\{ \DS(x) : x \in (\field^d)^\NOne_c\}$ as a semi-algebraic set? (Compare \cite{AmendolaFrizSturmfels2019} for the investigation of the image of iterated-integrals signatures as algebraic sets.)
  \item 
For $x \in (\field^d)^N$ denote by $\overleftarrow x$ the time series run backwards.
Then (as might surprise readers familiar with Chen's signature)
$\DS(\overleftarrow x) \centerdot \DS(x) \neq \varepsilon$. What are the implications?
  \item
The lead-lag procedure of \cite{Flint2016} lifts a discrete time series of dimension
$d$ to a piecewise smooth curve of dimension $2d$.
Since the resulting iterated-integrals signature
is invariant to time warping as well as space translations, and is polynomial in the original time series,
by Lemma \ref{lem:allInvariants} it must be contained in the iterated-integrals signature $\DS(x)$.
Conversely, is the signature of the resulting $2d$ curve
enough to recover the iterated-sum signature?
This would give a
\emph{finite dimensional} smooth curve whose iterated-integrals signature contains the invariants presented
in this paper (compare Theorem \ref{theorem:hoffmanDoesTheJob} for an \emph{infinite dimensional} smooth curve
doing the job).
\end{itemize}

\medskip

{\bf{Acknowledgments}} The authors would like to thank the referee for valuable remarks and comments, in particular with respect to reference \cite{KiralyOberhauser2019}. N.T.~kindly acknowledges support from the European Research Consortium for Informatics and Mathematics through post-doctral fellowship ERCIM 2018-10 and from the Excelence Cluster MATH+ EF1.



%
%
%
\bibliographystyle{ntplain}
\bibliography{discint}
\end{document}